\documentclass[a4paper]{amsart}

\usepackage[all,cmtip]{xy}
\usepackage{amssymb}
\usepackage{mathrsfs}
\usepackage{mathpazo}
\usepackage{setspace}
\usepackage{changepage}
\usepackage{euler}
\usepackage{leftidx}
\usepackage{hyperref}
\hypersetup{colorlinks}
\hypersetup{
    bookmarks=true,         
    unicode=false,          
    pdftoolbar=true,        
    pdfmenubar=true,        
    pdffitwindow=false,     
    pdfstartview={FitH},    
    pdftitle={My title},    
    pdfauthor={Author},     
    pdfsubject={Subject},   
    pdfcreator={Creator},   
    pdfproducer={Producer}, 
    pdfkeywords={keyword1} {key2} {key3}, 
    pdfnewwindow=true,      
    colorlinks=true,       
    linkcolor=red,          
    citecolor=magenta,        
    filecolor=violet,      
    urlcolor=blue      
}
\usepackage[OT2,T1]{fontenc}
\DeclareSymbolFont{cyrletters}{OT2}{wncyr}{m}{n}
\DeclareMathSymbol{\Sha}{\mathalpha}{cyrletters}{"58}

\newtheorem{thm}{Theorem}[section]

\newtheorem{lem}[thm]{Lemma}
\newtheorem{prop}[thm]{Proposition}

\theoremstyle{definition}

\makeatletter
\def\blfootnote{\xdef\@thefnmark{}\@footnotetext}
\makeatother

\makeatletter
\newcommand{\holim@}[2]{%
  \vtop{\m@th\ialign{##\cr
    \hfil$#1\operator@font holim$\hfil\cr
    \noalign{\nointerlineskip\kern1.5\ex@}#2\cr
    \noalign{\nointerlineskip\kern-\ex@}\cr}}%
}
\newcommand{\varinjholim}{%
  \mathop{\mathpalette\holim@{\rightarrowfill@\textstyle}}\nmlimits@
}
\newcommand{\varprojholim}{%
  \mathop{\mathpalette\holim@{\leftarrowfill@\textstyle}}\nmlimits@
}

\newcommand{\psdlim@}[2]{%
  \vtop{\m@th\ialign{##\cr
    \hfil$#1\operator@font Psdlim$\hfil\cr
    \noalign{\nointerlineskip\kern1.5\ex@}#2\cr
    \noalign{\nointerlineskip\kern-\ex@}\cr}}%
}
\newcommand{\varinjpsdlim}{%
  \mathop{\mathpalette\psdlim@{\rightarrowfill@\textstyle}}\nmlimits@
}
\newcommand{\varprojpsdlim}{%
  \mathop{\mathpalette\psdlim@{\leftarrowfill@\textstyle}}\nmlimits@
}

\DeclareMathOperator{\coker}{coker}

\DeclareMathOperator{\Spec}{Spec}

\makeatother

\DeclareMathAlphabet{\mathpzc}{OT1}{pzc}{m}{it}
\numberwithin{equation}{section}

\begin{document}

\title{BF path integrals for elliptic curves and $p$-adic $L$-functions}

\author{Jeehoon Park}
\address{Jeehoon Park: QSMS, Seoul National University, 1 Gwanak-ro, Gwanak-gu, Seoul, South Korea 08826 }
\email{jpark.math@gmail.com}

\author{Junyeong Park}
\address{Junyeong Park: Department of Mathematical sciences, Ulsan National Institute of Science and Technology, UNIST-gil. 50, Ulsan 44919, Korea}
\email{junyeongp@gmail.com}

\maketitle

\date{}



\begin{abstract} We prove an arithmetic path integral formula for the inverse $p$-adic absolute values of the $p$-adic $L$-functions of elliptic curves over the rational numbers with good ordinary reduction at an odd prime $p$ based on the Iwasawa main conjecture and Mazur's control theorem. This is an elliptic curve analogue of \cite{CCKKPY}.
\end{abstract}

\tableofcontents

\section{Introduction}
The arithmetic BF theory for number fields and abelian varieties was introduced in \cite{CaKi22} to show the philosophy of arithmetic gauge theory which indicates that the path integral of the physical theory is closely related to the $L$-function of the relevant number theory. Such a trial in \cite{CaKi22} was not a complete success but there was some hint towards such philosophy, which later led to  an arithmetic path integral formula for the inverse $p$-adic absolute values of Kubota-Leopoldt 
$p$-adic $L$-functions at roots of unity (a precise connection between Kubota-Leopoldt $p$-adic $L$-function and the arithmetic BF path integral of number fields up to $p$-adic units) in \cite{CCKKPY}. The formula in \cite{CCKKPY} is yet enough to realize the philosophy but \cite{CCKKPY} was more closer to it than \cite{CaKi22}. Our goal is to do the same job as \cite{CCKKPY} in the case of elliptic curves with ordinary good reduction at $p$. Though it turns out to be a simple exercise, it seems worthwhile to record such a formula in the literature.

\subsection{The statement of the main theorem}
Let $E/\mathbb{Q}$ be an elliptic curve with semistable reduction at all places. Let $p$ be an odd prime where $E$ has good ordinary reduction. For $n\geq0$, denote
\begin{align*}
K_n:=\mathbb{Q}(\zeta_{p^{n+1}}),\quad X_n:=\Spec\mathbb{Z}[\zeta_{p^{n+1}}]
\end{align*}
where $\zeta_{p^{n+1}}$ is a primitive $p^{n+1}$-th root of unity. Also, we simply denote $K:=K_0$ so that
\begin{align*}
\Gamma:=\mathrm{Gal}(K_\infty/K)\cong\mathbb{Z}_p.
\end{align*}
Denote $\Gamma_n\subseteq\Gamma$ the subgroup of index $p^n$. On the other hand, let
\begin{align*}
\xymatrix{\omega:\mathrm{Gal}(K/\mathbb{Q}) \ar[r] & \mathbb{Z}_p^\times(\cong\mathrm{Gal}(K_\infty/\mathbb{Q}))}
\end{align*}
be the Teichm\"uller character and let $(\cdot)_r$ be the $\omega^r$-isotypic component of a $\Gamma$-module with $\mathrm{Gal}(K/\mathbb{Q})\cong\mathbb{F}_p^\times$-action. Then $\mathbb{Q}_n:=K_{n,0}\subseteq K_n$ is the subfield fixed under $\omega$ so that we have the following diagram of field extensions:
\begin{align}
\begin{aligned}
\xymatrix@!0@R=2.5pc@C=4.5pc{
& K_n \\
\mathbb{Q}_n=K_{n,0}\ar@{-}[ur]^-{\mathbb{F}_p^\times} & \\
& K \ar@{-}[uu]_-{\Gamma/\Gamma_n} \\
\mathbb{Q} \ar@{-}[ur]_-{\mathbb{F}_p^\times} \ar@{-}[uu]^-{\Gamma/\Gamma_n} &
}
\end{aligned}
\end{align}
We also denote
\begin{align*}
Y_n:=\Spec\mathcal{O}_{\mathbb{Q}_n}
\end{align*}
Denote the N\'eron model of $E$ over $\mathbb{Z}$ by $\mathcal{E}$. For $m\geq1$ we use the following notation:
\begin{align*}
\mathscr{F}^m(\Spec\mathcal{O}):=H_\mathrm{fppf}^1(\Spec\mathcal{O},\mathcal{E}[p^m])\times H_\mathrm{fppf}^1(\Spec\mathcal{O},\mathcal{E}[p^m])
\end{align*}
where $\mathcal{O}$ is the ring of integers of a number field and we view $\mathcal{E}[p^m]$ as sheaves in the flat topology. Denote $\mathcal{E}^0\subseteq\mathcal{E}$ the identity component and $\Phi_\mathcal{E}$ the group of connected components. Then, as in \cite[p. 1305]{CaKi22}, we have an exact sequence:
\begin{align*}
\xymatrix{0 \ar[r] & \mathcal{E}^0 \ar[r] & \mathcal{E} \ar[r] & \Phi_\mathcal{E} \ar[r] & 0}
\end{align*}
Note that if $\mathbb{Z}\hookrightarrow\mathcal{O}$ is ramified only at $p\in\mathbb{Z}$, then $\mathcal{E}\otimes_\mathbb{Z}\mathcal{O}$ is the N\'eron model of $E$ over $\mathcal{O}$ and the order of $\Phi_{\mathcal{E}\otimes_\mathbb{Z}\mathcal{O}}$ is the same as the order of $\Phi_\mathcal{E}$ (see Lemma \ref{H1SelSha}). We make the following assumptions:
\begin{itemize}
\item the Tate-Shafarevich group $\Sha(K_n,E)$ is finite,
\item the order of $\Phi_\mathcal{E}$ is relatively prime to $p$, and
\item $E[p]$ is irreducible as a $\mathrm{Gal}(\overline{\mathbb{Q}}/\mathbb{Q})$-representation.
\end{itemize}
Note that by the first assumption,
\begin{align*}
\Sha(K_n,E)[p^m]=\Sha(K_n,E)[p^{2m}]\quad\textrm{for all sufficiently large $m$}.
\end{align*}
With the first two assumptions, one can define a 3-dimensional arithmetic BF theory \cite{CaKi22}.
The input data of such theory consists of 
\begin{itemize}
\item (spacetime) the scheme $Y_n=\Spec (\mathcal{O}_{\mathbb{Q}_n})$
\item (space of fields) the space $\mathscr{F}^m(Y_n)$
\item (action functional)
a BF-functional (see \eqref{BFdefn} and \eqref{Fmnr}):
\begin{align*}
\xymatrix{\mathrm{BF}:\mathscr{F}^m(Y_n) \ar[r] & \displaystyle\frac{1}{p^m}\mathbb{Z}/\mathbb{Z}}
\end{align*}
 for each $n\geq0$ and $m\geq1$.
\end{itemize}
Then the output of the theory is the following path integral:
\begin{align*}
\sum_{(a,b)\in\mathscr{F}^m(Y_n)}\exp(2\pi i\mathrm{BF}(a,b)).
\end{align*}

Since $E$ has good ordinary reduction at $p$, there exists a power series $g_E(t)\in\mathbb{Z}_p[[t]]\otimes_{\mathbb{Z}_p}\mathbb{Q}_p$ which represents the $p$-adic $L$-function $L_p(E/\mathbb{Q},s)$ of an elliptic curve (see \cite{MTT} and \cite[p. 459]{Gre01}). Under the assumption that $E[p]$ is an irreducible $\mathrm{Gal}(\overline{\mathbb{Q}}/\mathbb{Q})$-module, $g_E(t)\in\mathbb{Z}_p[[t]]$ holds.
Sometimes $g_E(t)$ is called an analytic $p$-adic $L$-function of $E$, while there is the notion of an algebraic $p$-adic $L$-function of $E$ defined by the generator of the characteristic ideal of the Pontrygin dual of the Selmer group $\mathrm{Sel}(\mathbb{Q}_\infty, E[p^\infty])$ which is a torsion $\mathbb{Z}_p[[t]]$-module.\footnote{The Selmer groups for each $n\in\mathbb{N}\cup\{\infty\}$ fit into the following exact sequences:
\begin{align*}
    \xymatrixcolsep{1.5pc}\xymatrix{0 \ar[r] & E(\mathbb{Q}_n)[p^\infty]\otimes_\mathbb{Z}\mathbb{Q}_p/\mathbb{Z}_p \ar[r] & \mathrm{Sel}(\mathbb{Q}_n,E[p^\infty]) \ar[r] & \Sha(\mathbb{Q}_n,E)[p^\infty] \ar[r] & 0}.
\end{align*}
We refer to \cite[chapter 2]{Gre01} for detailed definition.}
The Iwasawa main conjecture \cite[Conjecture 4.16]{Gre01}, which is now a theorem by Skinner-Urban \cite{SkiUrb}, asserts that they are the same.

Let $|\cdot|_p$ be the $p$-adic absolute value on the algebraic closure $\overline{\mathbb Q}_p$ normalized by $|p|_p=p^{-1}$.
Now we can state our main theorem.
\begin{thm}\label{mainthm} For a fixed $n\geq0$, suppose that the following additional assumptions hold:
\begin{itemize}
\item $E(\mathbb{Q}_n)$ is finite,
\item $E(\mathbb{Q}_n)[p^\infty]=0$, and
\item $\mathrm{Sel}(\mathbb{Q}_n,E[p^\infty])$ is finite.
\end{itemize}
For each $v\in Y_n$, let $c_v^{(p)}(E)$ be the highest power of $p$ dividing the Tamagawa factor $c_v(E)$ for $E$ at $v$. Then the following path integral formula hold.
\begin{align*}
&\quad\left|\prod_{\zeta^{p^n}=1}g_E(\zeta-1)\right|_p^{-1}\\
&=\left|\widetilde{E}(\mathbb{F}_p)[p^\infty]\right|^2\cdot\prod_{\substack{v\in Y_n \\ v\nmid p,v|N_E}}c_v^{(p)}(E)\cdot\lim_{m\rightarrow\infty}\sum_{(a,b)\in\mathscr{F}^m(Y_n)}\exp(2\pi i\mathrm{BF}(a,b))
\end{align*}
where $\widetilde{E}$ is the reduction of $E$ at $p$, and $N_E$ is the conductor of $E$.
\end{thm}
We prove this theorem in section \ref{mainthmpf}. For the proof, we first derive a path integral formula (Lemma \ref{keylemma}) in more general context, using Mazur's control theorem \cite[Theorem 4.1]{Gre01} and the Iwasawa main conjecture. Then we analyze the ``error term'' of Mazur's control theorem following the method of \cite{Gre99}.

\subsection{Open question}
Let $\alpha_p, \beta_p \in \overline{\mathbb{Q}}$ be defined by $\alpha_p+\beta_p=a_p$ and $\alpha_p\beta_p=p$, where $a_p= 1 + p -\tilde E(\mathbb{F}_p)$. Then $p$ does not divide $a_p$, which means $p$ splits in $\mathbb{Q}(\alpha_p, \beta_p)$. Let $\tau(\chi) \in \overline{\mathbb{Q}}$ be the Gauss sum for a Dirichlet character $\chi$. By the modularity of $E/\mathbb{Q}$, the $L$-value $L(E/\mathbb{Q}, \chi, 1)$ is defined and $L(E/\mathbb{Q}, \chi, 1)/\Omega_E$ is known to be algebraic by a theorem of Shimura, where $\Omega_E=\int_{E(\mathbb{R})} \frac{dx}{y}$. Let us fix a topological generator $\gamma_0 \in \mathrm{Gal}(\mathbb{Q}_\infty/\mathbb{Q})$. 
If $\chi$ is viewed as a faithful character of $\mathrm{Gal}(\mathbb{Q}_n/\mathbb{Q})$ with $n\geq 1$, then the conductor of $\chi$ is $p^{n+1}$ and $\zeta=\chi(\gamma_0)$ is a $p^n$-th root of unity. Now the interpolation property of $g_E(t)$ is given by
\begin{align*}
g_E(0) = \frac{(1-\beta_p p^{-1})^2 L(E/\mathbb{Q}, 1)}{\Omega_{E}} \\
g_E(\zeta-1) = \frac{(\beta_p)^{n+1} L(E/\mathbb{Q}, \chi, 1)}{ \tau(\chi)\Omega_{E}}
\end{align*}
for $n \geq 1$. The interesting open question is to enlarge the space of fields $\mathscr{F}^m(Y_n)$ or modify the BF-functional so that we can obtain a path integral formula for the $L$-value $\prod_{\zeta^{p^n}=1}g_E(\zeta-1)$ itself, which amounts to remove the $p$-adic absolute value from the formula in Theorem \ref{mainthm} incorporating $p$-adic unit information.

\subsection{Acknowledgement}

Jeehoon Park was supported by the National Research Foundation of Korea (NRF-2021R1A2C1006696) and the National Research Foundation of Korea grant (NRF-2020R1A5A1016126) funded by the Korea government (MSIT). Junyeong Park was supported by Samsung Science and Technology Foundation under Project Number SSTF-BA2001-02.

\section{The BF-functional for elliptic curves}
From now on, we use the following assumptions for each $n\geq0$:
\begin{itemize}
\item the Tate-Shafarevich group $\Sha(K_n,E)$ is finite, and
\item the order of $\Phi_\mathcal{E}$ is relatively prime to $p$.
\end{itemize}
We now recall the definition of the BF-functional in \cite[p.1303]{CaKi22}.
By \cite[Corollary 3.4]{Mil06}, we have a perfect pairing\footnote{By \cite[p.220]{Mil06}, we have $H_{\mathrm{fppf},c}^\bullet(X_n,-)\cong H_\mathrm{fppf}^\bullet(X_n,-)$.}:
\begin{align*}
\xymatrix{\cup:H_\mathrm{fppf}^1(X_n,\mathcal{E}[p^m])\times H_\mathrm{fppf}^2(X_n,\mathcal{E}[p^m]) \ar[r] & H_\mathrm{fppf}^3(X_n,\mathbb{G}_m)[p^m]}.
\end{align*}
together with an isomorphism as in \cite[p. 252]{Mil06}:
\begin{align*}
\xymatrix{\mathrm{inv}:H_\mathrm{fppf}^3(X_n,\mathbb{G}_m) \ar[r]^-\sim & \mathbb{Q}/\mathbb{Z}}
\end{align*}
which restricts to
\begin{align*}
\xymatrix{\mathrm{inv}:H_\mathrm{fppf}^3(X_n,\mathbb{G}_m)[p^m] \ar[r]^-\sim & \displaystyle\frac{1}{p^m}\mathbb{Z}/\mathbb{Z}}.
\end{align*}
Finally, let
\begin{align*}
\xymatrix{\delta:H_\mathrm{fppf}^1(X_n,\mathcal{E}[p^m]) \ar[r] & H_\mathrm{fppf}^2(X_n,\mathcal{E}[p^m])}
\end{align*}
be the Bockstein map coming from the exact sequence:
\begin{align*}
\xymatrix{0 \ar[r] & \mathcal{E}[p^m] \ar[r] & \mathcal{E}[p^{2m}] \ar[r]^-{p^m} & \mathcal{E}[p^m] \ar[r] & 0}.
\end{align*}
Combining all these, we define the BF-functional as follows:
\begin{align}\label{BFdefn}
\xymatrix{\mathrm{BF}:\mathscr{F}^m(X_n) \ar[r] & \displaystyle\frac{1}{p^m}\mathbb{Z}/\mathbb{Z} & (a,b) \ar@{|->}[r] & \mathrm{inv}(a\cup\delta b)}
\end{align}
\begin{lem}\label{H1SelSha} Let $\mathcal{O}$ be the ring of integers of a number field such that $\mathbb{Z}\hookrightarrow\mathcal{O}$ is ramified only at $p\in\mathbb{Z}$.
\begin{quote}
(1) $\mathcal{E}\otimes_\mathbb{Z}\mathcal{O}$ is the N\'eron model of $E$ over $\mathcal{O}$.\\
(2) The order of $\Phi_{\mathcal{E}\otimes_\mathbb{Z}\mathcal{O}}$ is the same as the order of $\Phi_\mathcal{E}$.\\
(3) We have the following isomorphisms:
\begin{align*}
H_\mathrm{fppf}^1(\Spec\mathcal{O},\mathcal{E}[p^m])\cong\mathrm{Sel}(\mathrm{Frac}(\mathcal{O}),E[p^m])
\end{align*}
\begin{align*}
H_\mathrm{fppf}^1(\Spec\mathcal{O},\mathcal{E})[p^m]\cong\Sha(\mathrm{Frac}(\mathcal{O}),E)[p^m]
\end{align*}
where $\mathrm{Frac}(\mathcal{O})$ is the field of fractions of $\mathcal{O}$.
\end{quote}
\end{lem}
\begin{proof} (1) follows because the \'etale base change of a N\'eron model is still a N\'eron model and our $E$ has good ordinary reduction at $p$.\\
\indent (2) Since $\mathbb{Z}\hookrightarrow\mathcal{O}$ is ramified only at $p\in\mathbb{Z}$, the number of connected component may vary at the primes in $\mathcal{O}$ lying over $p$. Since $E$ has good ordinary reduction at $p$, our $\mathcal{E}$ is always connected at these primes.\\
\indent (3) By (1) and (2), the first isomorphism comes from \cite[Lemma A.2]{CaKi22} and the second from \cite[Lemma A.3]{CaKi22}.
\end{proof}
\begin{prop} For every $n$ and every sufficiently large $m$, we have
\begin{align*}
\sum_{(a,b)\in\mathscr{F}^m(X_n)}\exp(2\pi i\mathrm{BF}(a,b))=|\mathrm{Sel}(K_n,E[p^m])|\left|\frac{E(K_n)}{p^mE(K_n)}\right|
\end{align*}
\end{prop}
\begin{proof} The assertion follows from \cite[Section 3]{CaKi22} by noting that $E$ is self-dual.
\end{proof}

\section{Isotypic components of the BF-functional}\label{BF-components}
The $\mathrm{Gal}(K_\infty/\mathbb{Q})$-action on $X_n$ induces by functoriality a $\mathrm{Gal}(K_\infty/\mathbb{Q})$-module structure on $H_\mathrm{fppf}^\bullet(X_n,\mathcal{E}[p^m])$. Since $p-1$ is relatively prime to $p$, the $\mathrm{Gal}(K/\mathbb{Q})$-action on $H_\mathrm{fppf}^\bullet(X_n,\mathcal{E}[p^m])$ is semisimple. Denote
\begin{align}\label{Fmnr}
\mathscr{F}^m_r(X_n):=H_\mathrm{fppf}^1(X_n,\mathcal{E}[p^m])_r\times H_\mathrm{fppf}^1(X_n,\mathcal{E}[p^m])_{-r}
\end{align}
so that we have
\begin{align*}
\mathscr{F}^m(X_n)=\bigoplus_{r=0}^{p-2}\mathscr{F}^m_r(X_n).
\end{align*}
We will show that $\mathscr{F}_0^{m}(X_n) = \mathscr{F}^m(Y_n)$.

By functoriality, $\cup$ and $\delta$ are $\omega^r$-equivariant. Hence they restrict to
\begin{align*}
\xymatrix{\cup:H_\mathrm{fppf}^1(X_n,\mathcal{E}[p^m])_r\times H_\mathrm{fppf}^2(X_n,\mathcal{E}[p^m])_s \ar[r] & H_\mathrm{fppf}^3(X_n,\mathbb{G}_m)_{r+s}}
\end{align*}
\begin{align*}
\xymatrix{\delta:H_\mathrm{fppf}^1(X_n,\mathcal{E}[p^m])_r \ar[r] & H_\mathrm{fppf}^2(X_n,\mathcal{E}[p^m])_r}.
\end{align*}
Since the $\mathrm{Gal}(K/\mathbb{Q})$-action on
\begin{align*}
\xymatrix{\mathrm{inv}:H_\mathrm{fppf}^3(X_n,\mathbb{G}_m) \ar[r]^-\sim & \displaystyle\frac{1}{p^m}\mathbb{Z}/\mathbb{Z}}
\end{align*}
is trivial, $H_\mathrm{fppf}^3(X_n,\mathbb{G}_m)_{r+s}\neq0$ if and only if $r+s\equiv0\bmod p-1$ so the BF-functional \eqref{BFdefn} splits into
\begin{align*}
\xymatrix{\displaystyle\sum_{r=0}^{p-2}\mathrm{BF}_r:\bigoplus_{r=0}^{p-2}\mathscr{F}^m_r(X_n) \ar[r] & \displaystyle\frac{1}{p^m}\mathbb{Z}/\mathbb{Z}}.
\end{align*}
Therefore, we have
\begin{align*}
\sum_{(a,b)\in\mathscr{F}^m(X_n)}\exp(2\pi i\mathrm{BF}(a,b))=\prod_{r=0}^{p-2}\sum_{(a,b)\in\mathscr{F}^m_r(X_n)}\exp(2\pi i\mathrm{BF}(a,b))
\end{align*}
\begin{prop}\label{Fmnr-integral} For every $n$ and every sufficiently large $m$, we have
\begin{align*}
\sum_{(a,b)\in\mathscr{F}^m_r(X_n)}\exp(2\pi i\mathrm{BF}(a,b))=|\mathrm{Sel}(K_n,E[p^m])_r|\left|\left(\frac{E(K_n)}{p^mE(K_n)}\right)_{-r}\right|
\end{align*}
\end{prop}
\begin{proof} If $\delta b\neq0$, then the sum over $a\in H_\mathrm{fppf}^1(X_n,\mathcal{E}[p^m])_{-r}$ becomes
\begin{align*}
\sum_a\exp(2\pi i\mathrm{BF}(a,b))=\sum_a\exp(2\pi\cdot\mathrm{inv}(\delta a\cup b))=0.
\end{align*}
On the other hand, if $\delta b=0$, then $\exp(2\pi i\mathrm{BF}(a,b))=1$. Since $\delta$ is $\omega$-equivariant,
\begin{align*}
(\ker\delta)_{-r}=H_\mathrm{fppf}^1(X_n,\mathcal{E}[p^m])_{-r}\cap\ker\delta.
\end{align*}
Combining these, we get
\begin{align*}
\sum_{(a,b)\in\mathscr{F}^m_r(X_n)}\exp(2\pi i\mathrm{BF}(a,b))=\left|H_\mathrm{fppf}^1(X_n,\mathcal{E}[p^m])_r\right||(\ker\delta)_{-r}|.
\end{align*}
For sufficiently large $m$ and $n$, we have a factorization coming from \cite[p.1304]{CaKi22}:
\begin{align*}
\xymatrixcolsep{-1.5pc}\xymatrix{
H_\mathrm{fppf}^1(X_n,\mathcal{E}[p^m]) \ar@{->>}[dr] \ar[rr]^-\delta & & H_\mathrm{fppf}^2(X_n,\mathcal{E}[p^m]) \\
& H_\mathrm{fppf}^1(X_n,\mathcal{E})[p^m] \ar@{^(->}[ur] &
}
\end{align*}
where the surjection fits into the Kummer sequence:
\begin{align*}
\xymatrix{0 \ar[r] & \displaystyle\frac{E(K_n)}{p^mE(K_n)} \ar[r] & H_\mathrm{fppf}^1(X_n,\mathcal{E}[p^m]) \ar[r] & H_\mathrm{fppf}^1(X_n,\mathcal{E})[p^m] \ar[r] & 0}
\end{align*}
Consequently,
\begin{align*}
(\ker\delta)_{-r}=\left|\left(\frac{E(K_n)}{p^mE(K_n)}\right)_{-r}\right|
\end{align*}
The other factor is determined from (3) of Lemma \ref{H1SelSha}.
\end{proof}
We conclude this section by realizing Proposition \ref{Fmnr-integral} as a path integral on $Y_n=X_{n,0}$. Since $(\mathcal{O}_{K_n})_0=\mathcal{O}_{\mathbb{Q}_n}$, the $\omega^r$-isotypic decomposition of $\mathcal{O}_{K_n}$ becomes
\begin{align*}
\mathcal{O}_{K_n}=\mathcal{O}_{K_{n,0}}\oplus\bigoplus_{r=1}^{p-2}(\mathcal{O}_{K_n})_r
\end{align*}
Then each factor in the above decomposition gives
\begin{align*}
H_\mathrm{fppf}^\bullet(X_n,\mathcal{E}[p^m])_r\cong H_\mathrm{fppf}^\bullet(X_{n,0},\mathcal{E}[p^m]\otimes_\mathbb{Z}(\mathcal{O}_{K_n})_r).
\end{align*}
In particular, we have the following isomorphisms:
\begin{align*}
H_\mathrm{fppf}^\bullet(X_n,\mathcal{E}[p^m])_0&\cong H_\mathrm{fppf}^\bullet(X_n,\mathcal{E}[p^m])^{\mathrm{Gal}(K/\mathbb{Q})}\\
&\cong H_\mathrm{fppf}^\bullet(X_{n,0},\mathcal{E}[p^m])\cong\mathrm{Sel}(\mathbb{Q}_n,E[p^m]).
\end{align*}
because taking the $\mathrm{Gal}(K/\mathbb{Q})$-invariant of abelian $p$-groups is an exact functor. Denote $D$ the Cartier dual of finite group schemes. Since $D$ commutes with the pushforward along finite \'etale ,maps (cf. \cite[Proposition D.1]{Zev}), we have
\begin{align*}
D(\mathcal{E}[p^m]\otimes_\mathbb{Z}(\mathcal{O}_{K_n})_r)\cong\mathcal{E}[p^m]\otimes_\mathbb{Z}(\mathcal{O}_{K_n})_{-r}
\end{align*}
Consequently, \eqref{Fmnr} can be rewritten as
\begin{align*}
\mathscr{F}^m_r(X_n)\cong H_\mathrm{fppf}^1(X_{n,0},\mathcal{E}[p^m]\otimes_\mathbb{Z}(\mathcal{O}_{K_n})_r)\times H_\mathrm{fppf}^1(X_{n,0},\mathcal{E}[p^m]\otimes_\mathbb{Z}(\mathcal{O}_{K_n})_{-r}).
\end{align*}
Moreover, since the pushforward along finite \'etale maps is exact by \cite[\href{https://stacks.math.columbia.edu/tag/03QP}{Tag 03QP}]{Stacks} \cite[\href{https://stacks.math.columbia.edu/tag/0DDU}{Tag 0DDU}]{Stacks}, and \cite[\href{https://stacks.math.columbia.edu/tag/0DDU}{Tag 0DDU}]{Stacks}, we have the corresponding Bockstein map:
\begin{align*}
\xymatrix{\delta:H_\mathrm{fppf}^1(X_{n,0},\mathcal{E}[p^m]\otimes_\mathbb{Z}(\mathcal{O}_{K_n})_{-r}) \ar[r] & H_\mathrm{fppf}^2(X_{n,0},\mathcal{E}[p^m]\otimes_\mathbb{Z}(\mathcal{O}_{K_n})_{-r})}.
\end{align*}
Therefore, Proposition \ref{Fmnr-integral} can be rewritten as a path integral on $Y_n$ as desired. In particular, we have $\mathscr{F}^m_0(X_n)=\mathscr{F}^m(X_{n,0})=\mathscr{F}^m(Y_n)$ and hence
\begin{align}\label{pathint-first-form}
\sum_{(a,b)\in\mathscr{F}^m(Y_n)}\exp(2\pi i\mathrm{BF}(a,b))=|\mathrm{Sel}(\mathbb{Q}_n,E[p^m])|\left|\frac{E(\mathbb{Q}_n)}{p^mE(\mathbb{Q}_n)}\right|.
\end{align}

\section{The proof of the main theorem}\label{mainthmpf}
In this section, we prove Theorem \ref{mainthm}. We begin with a lemma. Define abelian groups $A_n$ and $B_n$ via the following exact sequence:
\begin{align*}
\xymatrix{0 \ar[r] & A_n \ar[r] & \mathrm{Sel}(\mathbb{Q}_n,E[p^\infty]) \ar[r] & \mathrm{Sel}(\mathbb{Q}_\infty,E[p^\infty])^{\Gamma_n} \ar[r] & B_n \ar[r] & 0},
\end{align*}
where $\mathrm{Sel}(F, E[p^\infty])$ is the Selmer group over $F$ associated to $E[p^\infty]$.
By Mazur's control theorem \cite[Theorem 4.1]{Gre01}, $A_n$ and $B_n$ are finite $p$-groups whose orders are bounded as $n\rightarrow\infty$.
\begin{lem} \label{keylemma} For each $n\geq0$, the BF functional satisfies the following formula:
\begin{align*}
&\quad\left|\prod_{\zeta^{p^n}=1}g_E(\zeta-1)\right|_p^{-1}\\
&=\frac{1}{|E(\mathbb{Q}_n)[p^\infty]|}\frac{|B_n|}{|A_n|}\lim_{m\rightarrow\infty}\frac{1}{p^{m\cdot\mathrm{rank}_\mathbb{Z}(E(\mathbb{Q}_n))}}\sum_{(a,b)\in\mathscr{F}^m(Y_n)}\exp(2\pi i\mathrm{BF}(a,b)).
\end{align*}
\end{lem}
\begin{proof} The proof is based on the Iwasawa main conjecture \cite[Conjecture 4.16]{Gre01} and Mazur's control theorem.

Since $E(\mathbb{Q}_n)$ is a finitely generated abelian group, $E(\mathbb{Q}_n)_\mathrm{tors}$ is a finite abelian group. Hence, for every sufficiently large $m$, the composition
\begin{align*}
\xymatrix{E(\mathbb{Q}_n)[p^\infty]=E(\mathbb{Q}_n)[p^m] \ar@{^(->}[r] & E(\mathbb{Q}_n) \ar[r] & \displaystyle\frac{E(\mathbb{Q}_n)_\mathrm{tors}}{p^mE(\mathbb{Q}_n)_\mathrm{tors}}}
\end{align*}
is an isomorphism. Consequently, for sufficiently large $m$,
\begin{align*}
\left|\frac{E(\mathbb{Q}_n)}{p^mE(\mathbb{Q}_n)}\right|=p^{m\cdot\mathrm{rank}_\mathbb{Z}(E(\mathbb{Q}_n))}|E(\mathbb{Q}_n)[p^\infty]|.
\end{align*}
Hence \eqref{pathint-first-form} can be rewritten as
\begin{align*}\frac{1}{p^{m\cdot\mathrm{rank}_\mathbb{Z}(E(\mathbb{Q}_n))}}\sum_{(a,b)\in\mathscr{F}^m(X_{n,0})}\exp(2\pi i\mathrm{BF}(a,b))=|\mathrm{Sel}(\mathbb{Q}_n,E[p^m])|\cdot|E(\mathbb{Q}_n)[p^\infty]|.
\end{align*}

Denote $(\cdot)^\vee$ the Pontryagin dual of abelian groups and let
\begin{align*}
V_n:=\mathrm{Sel}(\mathbb{Q}_n,E[p^\infty])^\vee.
\end{align*}
Choose a topological generator $\gamma\in\Gamma$, which gives a topological ring isomorphism
\begin{align*}
\xymatrix{\Lambda:=\mathbb{Z}_p[[t]] \ar[r]^-\sim & \mathbb{Z}_p[[\Gamma]] & t \ar@{|->}[r] & \gamma-1}
\end{align*}
Along this isomorphism, we can write
\begin{align*}
\left(\mathrm{Sel}(\mathbb{Q}_\infty,E[p^\infty])^{\Gamma_n}\right)^\vee\cong\frac{V_\infty}{((t+1)^{p^n}-1)V_\infty}.
\end{align*}
Note that $V_\infty$ is a torsion $\Lambda$-module by \cite[Theorem 1.5]{Gre99}. Now we use the following assumption:
\begin{itemize}
\item $E[p]$ is irreducible as a $\mathrm{Gal}(\overline{\mathbb{Q}}/\mathbb{Q})$-representation.
\end{itemize}
Then $g_E(t)\in\Lambda$ by \cite[p. 459]{Gre01}, and $V_\infty$ has no nonzero pseudo-null $\Lambda$-submodule by \cite[Proposition 4.15]{Gre99}. Applying the structure theorem for $\Lambda$-modules \cite[Theorem 3.1]{Gre01} to $V_\infty$ (cf. \cite[Lemma 4]{Skin}) and using \cite[Exercise 3.7]{Gre01}, we get
\begin{align*}
\left|\frac{V_\infty}{((t+1)^{p^n}-1)V_\infty}\right|=\left|\frac{\Lambda}{(g_E(t),(t+1)^p-1)}\right|=u\cdot\prod_{\zeta^{p^n}=1}g_E(\zeta-1),\quad u\in\mathbb{Z}_p^\times.
\end{align*}
This in turn gives the following equalities:
\begin{align*}
|\mathrm{Sel}(\mathbb{Q}_n,E[p^\infty])|&=|V_n|\\
&=\frac{|A_n^\vee|}{|B_n^\vee|}\left|\frac{V_\infty}{((t+1)^{p^n}-1)V_\infty}\right|=\frac{|A_n|}{|B_n|}\cdot u\cdot\prod_{\zeta^{p^n}=1}g_E(\zeta-1).
\end{align*}
Since $\mathrm{Sel}(\mathbb{Q}_n,E[p^\infty])$, $A_n$, and $B_n$ are abelian $p$-groups, we conclude that
\begin{align*}
\frac{|B_n|}{|A_n|}|\mathrm{Sel}(\mathbb{Q}_n,E[p^\infty])|=\left|\prod_{\zeta^{p^n}=1}g_E(\zeta-1)\right|_p^{-1},
\end{align*}
which finishes the proof.
\end{proof}
Let us return to the proof of Theorem \ref{mainthm}. By assumption, $\mathrm{rank}_\mathbb{Z}(E(\mathbb{Q}_n))=0$, and $|E(\mathbb{Q}_n)[p^\infty]|=1$. Hence, in view of Lemma \ref{keylemma}, it remains to determine $|A_n|$ and $|B_n|$ in Lemma \ref{keylemma}. Following \cite[section 3]{Gre99}, we begin with the following commutative diagram with exact rows:
\begin{align*}
\xymatrix{
0 \ar[r] & \mathrm{Sel}(\mathbb{Q}_n,E[p^\infty]) \ar[d]^-{s_n} \ar[r] & H^1(\mathbb{Q}_n,E[p^\infty]) \ar[d]^-{h_n} \ar[r] & \mathcal{G}_E(\mathbb{Q}_n) \ar[d]^-{g_n} \ar[r] & 0 \\
0 \ar[r] & \mathrm{Sel}(K_{n,\infty},E[p^\infty])^{\Gamma_n} \ar[r] & H^1(K_{n,\infty},E[p^\infty])^{\Gamma_n} \ar[r] & \mathcal{G}_E(K_{n,\infty})
}
\end{align*}
By our assumption and the proof of \cite[Lemma 3.1]{Gre99}, we have
\begin{align*}
    |\ker h_n|=|E(\mathbb{Q}_n)[p^\infty]|=1.
\end{align*}
From the snake lemma together with \cite[Lemma 3.2]{Gre99}, we get
\begin{align*}
    \frac{|B_n|}{|A_n|}=\frac{|\coker s_n|}{|\ker s_n|}=\frac{|\ker g_n|}{|\ker h_n|}=|\ker g_n|.
\end{align*}
Since $\Gamma_n\cong\mathbb{Z}_p$, we may apply \cite[Lemma 4.7]{Gre99} to the above diagram. Using \cite[Lemma 3.3]{Gre99}, \cite[Lemma 3.4]{Gre99}, and \cite[Proposition 4.8]{Gre99}, we get
\begin{align*}
|\ker g_n|=\left|\widetilde{E}(\mathbb{F}_p)[p^\infty]\right|^2\cdot\prod_{\substack{v\in Y_n \\ v\nmid p,v|N_E}}c_v^{(p)}(E).
\end{align*}
Note that $\mathbb{Q}_n/\mathbb{Q}$ is totally ramified at $p$ so the unique prime in $Y_n$ lying over $p$ has the residue field $\mathbb{F}_p$. This concludes the proof of Theorem \ref{mainthm}.

\end{document}